\documentclass[12pt,a4paper]{amsart}
\usepackage{fourier}
\usepackage{ulem}
\usepackage{amsthm}
\usepackage{mathrsfs}
\usepackage{graphicx}
\usepackage{stmaryrd}
\usepackage{float}
\usepackage{amsmath}
\usepackage[all]{xy}\usepackage{xypic}
\usepackage{hyperref}
\hypersetup{
  pdftitle   = {},
  pdfauthor  = {},
  pdfcreator = {\LaTeX\ with package \flqq hyperref\frqq}
}

\DeclareMathAlphabet{\mathpzc}{OT1}{pzc}{m}{it}
\usepackage{tikz}

\newtheorem{theorem}{Theorem}[section]
\newtheorem*{theorem*}{Theorem}
\newtheorem{proposition}[theorem]{Proposition}
\newtheorem{lemma}[theorem]{Lemma}
\newtheorem*{lemma*}{Lemma}
\newtheorem{corollary}[theorem]{Corollary}
\newtheorem*{corollary*}{Corollary}

\newtheorem*{conjecture*}{Conjecture}

\theoremstyle{definition}

\newtheorem{example}[theorem]{Example}
\newtheorem{notation}[theorem]{Notation}
\theoremstyle{remark}
\newtheorem{remark}[theorem]{Remark}

\DeclareMathOperator{\tr}{tr}

\DeclareMathOperator{\Diag}{Diag}

\DeclareMathOperator{\adj}{adj}
\DeclareMathOperator{\id}{id}
\numberwithin{equation}{section}

\begin{document}
\title{Some Landau--Ginzburg models   viewed as rational maps}

\author{  E. Ballico, E.  Gasparim, L. Grama,  {\tiny and} L.   San Martin}

\begin{abstract}\cite{GGSM2} showed  that  height functions give  adjoint orbits of semisimple Lie algebras 
 the structure of symplectic Lefschetz fibrations  (superpotential of 
the LG model in the language of mirror symmetry). 
We describe  how to extend  the superpotential to compactifications.
Our results explore the geometry of the adjoint orbit from 2 points of view: algebraic geometry and Lie theory.
\end{abstract}

\maketitle
\tableofcontents

\section{Motivation and statement of results}
The data  $W \colon X \rightarrow \mathbb C$  of a  manifold together with a complex function 
is commonly known in the literature as a {\it Landau--Ginzburg model};  $W$ is called the {\it superpotential}.
Such LG models are fundamental ingredients to the 
study of  questions motivated by the Homological Mirror Symmetry conjecture of Kontsevich \cite{Ko}. 
When the  superpotential is defined  over a symplectic manifold,
this involves the
construction of a Fukaya--Seidel category.  The objects of this category 
are associated to the singularities of $W$ (Lagrangian vanishing cycles), see \cite{Se}.
These questions motivated us   to construct examples of LG models
and to study their symplectic geometry  \cite{GGSM1}, \cite{GGSM2}. Here we explore these examples from 
different points of view, namely, we
are interested in the algebraic geometry and Lie theory of the smooth variety $X$. 

Part of the homological mirror symmetry   conjecture 
describes a duality between algebraic varieties and  symplectic LG models. 
 Subsequently Clarke \cite{Cl} showed that the conjecture can be interpreted
 in further generality in such a way that both sides are LG models.  
According to  whether $X$ is considered with a  K\"ahler structure 
and  algebraic potential or whether $X$ is considered with a symplectic structure
and  holomorphic potential these are called {\it $B$-side LG model}, 
or {\it $A$-side LG model}, see \cite {AAK}  or \cite{KKP}. In such terminology, the work of \cite{GGSM1} provides
a large family of examples  of 
$A$-side  LG models. 
 Here we look at these examples from the point of view of   $B$-side LG models. 
 
 Our constructions are carried out  
 using Lie theory.
Let $\mathfrak g$ be a complex semisimple Lie algebra with Lie group $G$,  and $\mathfrak h$ the Cartan subalgebra. 
Consider the adjoint orbit  $\mathcal O(H_0)$  of an element $H_ 0 \in \mathfrak h$, 
that is, $$\mathcal O(H_0) := \{\mathrm{Ad} (g) H_0, g \in G\}\text{.}$$ Let $H  \in \mathfrak h$ be a regular element, and 
 $\langle .,. \rangle$ the Cartan--Killing form. 
The main result of \cite{GGSM1} shows that the height function 
$$ \begin{array}{rll} f_H\colon  \mathcal O(H_0) & \rightarrow & \mathbb C\cr
                                             X & \mapsto & \langle X,H \rangle
                                            \end{array}$$
gives the orbit the structure of  a symplectic Lefschetz fibration; thus corresponding to an $A$-side LG model.  
We show here that this height function can also be interpreted as a rational map on a projective compactification of $\mathcal O(H_0)$;
hence corresponding to  a $B$-side LG model. 

In this work we  restrict ourselves to the class of adjoint orbits  which are diffeomorphic 
to cotangent bundles of projective spaces $\mathbb P^n$. This is the simplest case of semisimple orbit, 
yet already presenting somewhat surprising features. A harmonious combination of Lie theory and algebraic geometry
happens naturally in this context, for example,  we shall see that  Lie theory provides rather 
efficient methods  to identify the Segre embedding of a compactification of $\mathcal O(H_0)$. 
This method of carrying along  Lie theory together with algebraic geometry is 
arguably where the core value of our contribution  lies. We put forth the idea that there is much to profit 
from applying Lie theoretical methods to algebraic geometric problems. This work is   a first instance
of what we propose as a long term program. Certainly such combinations of the 2 areas have appeared in the literature
in other contexts, the
particularly new features of our contribution are  the  applications to the study Lefschetz fibrations and LG models.

Our  results  go as follows. 
We take   $G=SL(n+1,\mathbb{C)}$ and consider the adjoint orbit passing through  $\mu=\Diag(n,-1,
\ldots,-1)$, which we denote by  $\mathcal{O}_\mu$. The diffeomorphism type is then $\mathcal{O}_\mu  \sim_{dif} T^* \mathbb P^n$.
In  Section \ref{slfs} we recall the main result of \cite{GGSM1} showing that the height function  with respect to a regular element gives 
the adjoint orbit the structure of a symplectic Lefschetz fibration. 
We then describe the orbit and the regular fibers  of $f_H$ as affine varieties, and consider fibrewise compactifications. In section \ref{alg}  for the case of $\mathfrak {sl}(2, \mathbb C)$ we obtain:
 
 \begin{theorem*} {\bf \ref{sl2}} Let
$X=\mathcal{O}_{(1,-1)}$, $H=\Diag(1,-1)$ and  $f_H: X\to \mathbb{C}$. Then $f_H$ admits a fiberwise compactification 
with fibres isomorphic to  $\mathbb P^1$.
\end{theorem*}

In section \ref{inc} we present several ways to interpret the adjoint orbit, thus illustrating the interactions between 
Lie theory and algebraic geometry. In section  \ref{zar}  we 
describe (Zariski)  open charts  for  $\mathcal{O}_\mu$ in terms of  Bruhat cells.   

\begin{corollary*}{\bf \ref{bru}}
The domains of the parametrizations  $D_j$ corresponding to  the  Bruhat cells are open 
and dense in  $\mathcal{O}_\mu$.
\end{corollary*}

The orbit $\mathcal{O}_\mu$ is not  compact, thus we must choose  a 
  compactication. Once again we are faced with the decision whether to use
 Lie theory or  algebraic geometry.
Recently in  \cite{GGSM2} a holomorphic open and dense embbeding of  $\mathcal{O}_\mu$ into $\mathbb{F}\times \mathbb{F}^*$ 
was constructed.  
Here,  $\mathbb{F}$ and $\mathbb{F}^*$ represent a flag manifold and its  dual flag,
chosen such that   $\mathcal{O}(H_0) \sim_{dif}  T^*\mathbb{F}$ is a diffeomorphism.
The immediate task that then follows is to extend the potential $f_H$ to this compactification. 
Such an extension can not be made holomorphically, as explained in 
 lemmas \ref{pot1} and \ref{pot2}. We then proceed to extend   the potential as a rational function.
We consider here  $\mathcal O_\mu$ as the adjoint orbit of  $e_1\otimes \varepsilon_1$ 
in $\mathbb C^n \times (\mathbb C^n)^*$. Set $V = \mathbb C^n$.

\begin{theorem*}{\bf \ref{rat}}
 The rational function on  $V\otimes V^{\ast }$ that
coincides with  the potential $f_{H}$ on $\mathcal{O}\left( v_{0}\otimes \varepsilon
_{0}\right) $ 
is given by 
\begin{equation*}
R_{H}\left( A\right) =\frac{\mathrm{tr}\left( A\rho _{\mu }\left( H\right)
\right) }{\mathrm{tr}\left( A\right) }
\end{equation*}%
for $A\in V\otimes V^{\ast }=\mathrm{End}\left( V\right) $.
\end{theorem*}

  On the algebraic geometric side,  \cite{BCG} compactified  the affine variety $X = \mathcal{O}_\mu$ 
to a projective variety $\overline{X}$ by homogenising its defining ideal. Then the  question begging to be asked is whether 
 their algebraic compactification agrees with our Lie theoretical one. 
Using methods of computational algebraic geometry and a Macaulay2 algorithm \cite{BG} identified a projective embbeding of 
$\overline{X}$ as the Segre embbeding for the cases of $\mathfrak{sl}(n+1)$ with $n<10$, and conjectured that the result 
holds true for all $n$.  We provide an affirmative answer to this question, in particular concluding that 
the Lie theoretical compactification does coincide with the algebraic geometric one for the case of $\mathcal O_\mu$.

\begin{theorem*}{\bf \ref{seg}}
The embbeding  $\mathcal{O}_\mu \hookrightarrow \mathbb{P}^n\times  {\mathbb{P}^n}^*$ obtained by Lie 
theoretical methods agrees with the Segre embbeding obtained algebro-geometrically by homogenisation of the ideal cutting out 
the orbit $\mathcal O_\mu$ as an affine variety in $\mathfrak{sl}(n+1)$.
\end{theorem*}

\begin{remark}
Observe that the algebraic geometric method will in general produce singular compactifications, see \cite[Sec. 6]{BCG}, whereas that the Lie theoretical method always embeds the orbit into a product of smooth flag manifolds.  
\end{remark}

We conclude the paper by presenting in section \ref{seg} the expressions of the Segre embbeding and the rational potential $R_H$ 
first for the case  $n=3$, with $\mu=(2,-1,-1)$ hence  $\mathcal{O}_\mu\approx T^* \mathbb{P}^2$,
and finally in general for $\mu = (n, -1, \dots, -1)$ when  $\mathcal{O}_\mu\approx T^* \mathbb{P}^n.$

\section{Symplectic Lefschetz fibrations on adjoint orbits}\label{slfs}

Let $\mathfrak{g}$ be a complex semisimple Lie algebra with 
 Cartan--Killing form  $\langle X,Y\rangle =\mathrm{tr}%
\left( \mathrm{ad}\left( X\right) \mathrm{ad}\left( Y\right) \right) \in 
\mathbb{C}$,
and $G$ a connected
Lie group with Lie algebra $\mathfrak{g}$ .
Let $H_0\in \mathfrak h$. The adjoint orbit of $H_0$  is defined as 
\begin{equation*}\label{def} 
\mathcal{O}\left( H_0 \right) =\mathrm{Ad} G\cdot H_0=\{\mathrm{Ad} (g) H_0 \in \mathfrak{g}:g\in G\}. 
\end{equation*}%

Fix a Cartan subalgebra $\mathfrak{h}\subset \mathfrak{g}$ and a compact
real form  $\mathfrak{u}$ of $\mathfrak{g}$. Associated to these subalgebras
there are the subgroups $T=\langle \exp \mathfrak{h}\rangle =\exp \mathfrak{h%
}$ and $U=\langle \exp \mathfrak{u} \rangle =\exp \mathfrak{u}$. Denote by $%
\tau$ the conjugation associated to  $\mathfrak{u}$, defined by $\tau \left(
X\right) =X$ if $X\in \mathfrak{u}$ and $\tau \left( Y\right) =-Y$ if $Y\in i%
\mathfrak{u}$. Hence if $Z=X+iY\in \mathfrak{g}$ with $X,Y\in \mathfrak{u}$
then $\tau \left( X+iY\right) =X-iY$. In this case,  $%
\mathcal{H}_{\tau }:\mathfrak{g}\times \mathfrak{g}\rightarrow \mathbb{C}$ 
defined by 
\begin{equation}  \label{hermitian}
\mathcal{H}_{\tau }\left( X,Y\right) =-\langle X,\tau Y\rangle
\end{equation}
is a Hermitian form on $\mathfrak{g}$ (see \cite[lemma 12.17]{amalglie}).
We write the real and imaginary parts of $\mathcal{H}$ as 
\begin{equation*}
\mathcal{H}\left( X,Y\right) =\left( X,Y\right) +i\Omega \left( X,Y\right)
\qquad X,Y\in \mathfrak{g}. 
\end{equation*}%
The real part $\left( \cdot ,\cdot \right) $ is an inner product 
and the imaginary part of 
$\Omega $ is a symplectic form on $\mathfrak{g}$. Indeed, we have 
\begin{equation*}
0\neq i\mathcal{H}\left( X,X\right) =\mathcal{H}\left( iX,X\right) =i\Omega
\left( iX,X\right) , 
\end{equation*}%
that is, $\Omega \left( iX,X\right) \neq 0$ for all $X\in \mathfrak{g}$,
which shows that $\Omega $ is nondegenerate. Moreover, $d\Omega =0$ because $%
\Omega $ is a constant bilinear form.
The fact that $\Omega \left( iX,X\right) \neq 0$ for all $X\in \mathfrak{g} $
guarantees that the restriction of $\Omega $ to any complex subspace of $%
\mathfrak{g}$ is also nondegenerate.

Now, the tangent spaces to $\mathcal{O}\left( H_{0}\right) $ are complex
vector subspaces of $\mathfrak{g}$. Therefore, the pullback of $\Omega $ by
the inclusion $\mathcal{O}\left( H_{0}\right) \hookrightarrow \mathfrak{g}$
defines a symplectic form on  $\mathcal{O}\left( H_{0}\right) $.
With this choice of symplectic form,  the main result of \cite{GGSM1} says:

\vspace{5mm}

\noindent \textbf{Theorem} \cite[Thm. 2.2]{GGSM1} Let $\mathfrak{h}$ be the
Cartan subalgebra of a complex semisimple Lie algebra. Given $H_{0}\in 
\mathfrak{h}$ and $H\in \mathfrak{h}_{\mathbb{R}}$ with $H$ a regular
element. The \textit{height function} \ $f_{H}\colon\mathcal{O}\left(
H_{0}\right) \rightarrow \mathbb{C}$ defined by 
\begin{equation*}
f_{H}\left( x\right) =\langle H,x\rangle \qquad x\in \mathcal{O}\left(
H_{0}\right) 
\end{equation*}%
has a finite number (= $|\mathcal{W}|/|\mathcal{W}_{H_{0}}|$) of isolated
singularities and gives $\mathcal{O}\left(H_{0}\right) $ the structure of a
symplectic Lefschetz fibration. \vspace{3mm}

\begin{remark}
Note that the symplectic form $\Omega$ used here is not equivalent to the 
Kostant--Kirilov--Souriaux form on $\mathcal  O(H_0)$.
\end{remark}

\section{Adjoint orbits as algebraic varieties}\label{alg}

We now consider the case when $\mathfrak g = \mathfrak{sl}(n)$.
To write down the adjoint orbit  as an algebraic variety we can simply use 
the minimal polynomial corresponding to the  diagonal matrix $H_0$. Sometimes this method is not very 
economical, as 
it may  give more equations than needed. In fact, using the entries of the minimal polynomial 
results in cutting out  the adjoint orbit by $n^2$ equations inside the lie algebra $\mathfrak{sl}(n)$ which 
has dimension $n^2 -1$. Thus, for instance  we will certainly  have far too many equations whenever 
the orbit is a complete intersection. Nevertheless, there is the advantage that this 
method works in all cases and is easily programable into a computer algebra algorithm. 

Once we have described the orbit as an affine variety, we then wish to compactify it and to identify the 
closure of a regular fibre  under this compactification. 

As an example, we  discuss the case of  $\mathfrak{sl}(2,\mathbb C)$. Take
\[
  H
  = 
  H_0
  =
  \begin{pmatrix}
    1 & 0 \\
    0 & -1
  \end{pmatrix}.
\]
The Weyl group $\mathcal W
\simeq S_2$ acts via conjugation by permutation matrices.  The two
singularities  of the potential are thus $H_0$ and $-H_0$.  
In this section we prove:
 \begin{theorem}\label{sl2} Let
$X=\mathcal{O}_{(1,-1)}$, $H=\Diag(1,-1)$ and  $f_H: X\to \mathbb{C}$. Then $f_H$ admits a fiberwise compactification 
with fibres isomorphic to  $\mathbb P^1$.
\end{theorem}

We  describe the orbit as an affine variety embedded in $\mathbb C^3$.
Writing a
general element $A \in \mathcal O (H_0)$ as
\[
  A
  =
  \begin{pmatrix}
    x & y \\
    z & -x
  \end{pmatrix},
\]
the characteristic polynomial of $A$ is
\[-
  \left( x - \lambda \right) \left( x + \lambda \right) - yz
  =
  \det \left( A - \lambda I \right)
  =
  \lambda^2 - 1.
\]
 This  implies that
the orbit $\mathcal O (H_0) \subset \mathfrak{sl} (2, \mathbb C) \simeq
\mathbb C^3$ is an affine variety $X$ cut out by the equation
\begin{equation}
  \label{eq:sl2orbit}
  x^2 + yz - 1 = 0.
\end{equation}
We can compactify this variety by homogenising eq.~\ref{eq:sl2orbit} and embedding
$X$ into the corresponding projective variety. This produces the surface cut out by  
$x^2+yz-t^2=0$ in $\mathbb P^3$.

The height function on  $X=\mathcal O (H_0)$  
\[
  f_H (A)
  =
  \tr HA
  =
  \tr
  \begin{pmatrix}
    1 & 0 \\
    0 & -1
  \end{pmatrix}
  \begin{pmatrix}
    x & y \\
    z & -x
  \end{pmatrix}
  =
  2x,
\]
 has critical values $\pm 2$. 

  Thus, $0$ is a regular value, and we 
express the regular fibre  $X_0$  as the affine variety in
$\{(y,z) \in \mathbb C^2\}$ cut out by the equation
\[
  yz-1 = 0,
\]
since it must satisfy eq.~\ref{eq:sl2orbit} and $x = 0$.  As with the orbit,
we  homogenise this equation and embed the fibre into the corresponding
projective variety $\overline{X_0}$ cut out by the equations $x=0$ and $yz-t^2=0$ in $ \mathbb P^3$.  
Consider  the natural embedding given by 
$$\begin{array} {rrcl} i\colon & X_0 &\rightarrow & \overline{X_0} \subset \mathbb P^3\\
                                             & (y,z) & \mapsto & [0,y,z,1] .
                                              \end{array}$$
Claim: $\overline{X_0} \simeq \mathbb P^1$ and  $\overline{X_0}\setminus i(X_0) =  \{[0,1,0,0], [0,0,1,0]\}$. 
\begin{proof}
In fact, $\overline{X_0}$ is cut out in $\mathbb P^2$ by a degree 2 polynomial, so the isomorphism with $\mathbb P^1$ follows immediately from the degree-genus formula $g=(d-1)(d-2)/2$.

For the second part, given a point  $P =[0,y,z,t] \in \mathbb P^3$,  we have that $P \in  \overline{X_0}$ if and only if 
satisfies the equation $yz=t^2$. We have two possibilities:
\begin{itemize}
\item If $t\neq 0$, then  $yz\neq 0$ and we may rewrite  $P=[0,y,\frac{t^2}{y},t]= [0,\frac{y}{t},\frac{t}{y},1]$ and we have that 
$P=i(y',z')$ for 
$y'=\frac{y}{t}$ and $z'=\frac{t}{y}$. 
\item Otherwise $t=0$, and there are two such points in $ \overline{X_0}$, they are $[0,1,0,0]$ and $[0,0,1,0]$, neither of which belong to 
$i(X_0)$. 
\end{itemize}
\end{proof}

\section{Several incarnations of the orbit}\label{inc}

  There are various ways to interpret the adjoint orbit 
$\mathcal O(H_0)$  
depending on the point of view  best suited to a given problem. 
\begin{itemize}
\item  By definition \ref{def} the adjoint orbit is contained in the Lie algebra $\mathfrak{g}$, and 
 consists of all the points $\mathrm{Ad} g \cdot H_0$ with $g \in G$.

\item $\mathrm{Ad}\left( G\right) \cdot H_{0}$ can also be interpreted as a quotient of the 
group $G$, identifying  $\mathcal O(H_0)$   with the homogeneous
space $G/Z_{H_{0}}$ where $Z_{H_0}$ is the centralizer of $H_0$.  

\item  Take the simple roots of $\mathfrak g$ that have $H_0$ in their 
kernel,  denote by $\mathfrak p_0$  the  parabolic subalgebra they generate,  and by $P_0$  the  corresponding subgroup.
The compact subgroup $K$ of $G$ cuts out 
 the subadjoint orbit $\mathrm{Ad}\left( K\right) \cdot H_{0}$, which can be identified  with the flag manifold $\mathbb{F}_{H_{0}}=G/P_{H_{0}}$.
\cite[Sec. 2.2]{GGSM2} showed that $\mathcal O(H_0)$ is diffeomorphic to the cotangent bundle $T^*\mathbb{F}_{H_{0}}.$

\item From a Riemannian point of view this is also diffeomorphic to the tangent bundle  $T\mathbb{F}_{H_{0}}$.

\item In \cite{GGSM1} the  orbit $\mathcal{O} (H_{0})$ is given  the symplectic structure taken form  the imaginary part 
of the Hermitian form inherited from $\mathfrak g$. With this choice, the flag $\mathbb{F}_{H_{0}}$ is the Lagrangian in $\mathcal O(H_0)\simeq 
T^*\mathbb{F}_{H_{0}}$  and corresponds to the zero section of the cotangent bundle. 

\item As seen in section \ref{alg} the adjoint orbit is an affine algebraic variety.

\item  \cite[sec. 4.2]{GGSM2} showed that 
$\mathcal O(H_0)$ can be identified with the open  orbit of the diagonal action of $G$ on the product 
$\mathbb{F}_{H_{0}}\times \mathbb{F}_{H_{0}^*}$ . 
\end{itemize}

\begin{example} Consider the case when $\mathfrak g = \mathfrak{sl}(n+1)$ and $$H_0 = \Diag (n,-1, \dots, -1).$$  
Then the corresponding parabolic subgroup  is 
$$P_{H_0} = \{A \in \mathfrak{sl}(n+1) : a_{i1}= 0, \,\, \text{for all } 1<i \leq n+1\}. $$
Hence, $P_{H_0}$ consists of matrices of the form
$$\left(\begin{matrix} 
* & * & \cdots & * \\
0 & * & \cdots & * \\
\vdots  &  & \ddots \\
0 & *   & \cdots & *
\end{matrix}\right)$$
where $*$ denotes any complex number.

The centralizer  of $H_0$ is the subgroup 
$$Z_{H_0} = \{A \in \mathfrak{sl}(n+1) : a_{i1}= 0= a_{1i} \,\, \text{for all } 1<i \leq n+1\}. $$
Hence, $Z_{H_0}$ consists of matrices of the form
$$\left(\begin{matrix} 
* & 0  & \cdots & 0 \\
0 & * & \cdots & * \\
\vdots  &  & \ddots \\
0 & *   & \cdots & *
\end{matrix}\right).$$
 Here $F_{H_0} = \mathbb P^n$ and 
 $\mathcal O(H_0) \sim_{dif} T^*\mathbb P^n.$

\end{example} 

Different choices of $H_0 \in \mathfrak{sl}(n+1)$ will  in general  lead to different flag manifolds $F_{H_0}$. 
The flags range from the one of largest dimension, 
 the full flag  $F(1,2, \dots, n)$
 of subspaces   $V_1 \subset V_2 \subset \cdots \subset V_n \subset \mathbb C^{n+1}$ with 
$\dim V_i = i$ down  to  the case of only subspaces of dimension 1, namely the projective space $\mathbb P^n$.
The latter is usually called the minimal flag. The full flag occurs when all eigenvalues of $H_0$ are 
distinct, whereas the minimal flag occurs for $H_0 = \Diag  (n, -1, \dots, -1)$. 

\begin{notation}
We denote by $\mathcal O_\mu$ the adjoint orbit  $\mathcal O(H_0)$ associated to the matrix 
$ H_0 = \Diag  (n, -1, \dots, -1) \subset \mathfrak{sl}(n+1)$. Thus, this is the orbit of the {\it  minimal flag} $\mathbb P^n$ of $\mathfrak{sl}(n+1)$.
\end{notation}
In what follows we will restrict our attention to  orbits $\mathcal O_\mu$. For this case we have yet another incarnation 
of $\mathcal O(H_0)$, namely as the adjoint orbit of  $e_1\otimes \varepsilon_1$ 
in $\mathbb C^n \times (\mathbb C^n)^*$. 
To verify this, first notice that  $e_1 \otimes \epsilon_1$ 
corresponds to the matrix with  $1$ 
in the  $(1,1)$ entry and 
zeroes elsewhere. Hence, as a linear operator, the image is $1$-dimensional (generated by $e_1$) and 
the kernel is $n-1$-dimensional. The action of $G$ emcompasses all matrices with the same spectrum. 
The identification with the orbit  $\mathcal O_\mu$ 
is made using the moment map, as described in detail in \cite[Sec. 4]{GGSM2}.
It identifies the eigenspace associated to $0$  in $g \cdot (e_1\otimes \epsilon_0)$
with the eigenspace associated to the eigenvalue $n$ 
of $\mathrm{Ad} G H_0 \in \mathcal O_\mu$,
and analogously 
identifies the eigenspaces of  dimension $ n-1$. 

\section{Topology on  $\mathcal O_\mu$ }\label{zar}

We wish to cover $\mathcal O_\mu$  by open sets which are natural from the Lie theory point of view; these will turn out to 
be  Zariski open as well. To do so, we
use the model for the adjoint orbit of $e_1\otimes \varepsilon_1$ 
in $\mathbb C^n \times (\mathbb C^n)^* = \mathrm{End}(\mathbb C),$
which is the set of projections from $\mathbb C^n$ to subspaces of dimension 1. 
This model is  similar to the one where we  the orbit is viewed inside   the product 
$\mathbb P^{n-1}\times Gr_{n-1}(n)$, given as the set of pairs 
$([v],V)$ such that $v \notin V$. 
Such a pair corresponds to a projection over $[v]$ with kernel $V$.

Given $H_{0}\in \mathfrak{h}_{\mathbb{R}}$, let
\begin{equation*}
\mathfrak{n}_{H_{0}}^{+}=\sum_{\alpha \left( H_{0}\right) >0}\mathfrak{g}%
_{\alpha }\qquad \mathrm{and}\qquad \mathfrak{n}_{H_{0}}^{-}=\sum_{\alpha
\left( H_{0}\right) <0}\mathfrak{g}_{\alpha }
\end{equation*}%
be the sums of the eigenspaces of  $\mathrm{ad}\left( H_{0}\right) $
associated to  positive and  negative eigenvalues, respectively. The 
subspaces $\mathfrak{n}_{H_{0}}^{\pm }$ are 
nilpotent subalgebras. 
The corresponding subgroups  $N_{H_{0}}^{\pm }=\exp \mathfrak{n}_{H_{0}}^{\pm }$ are 
closed in $G$ and 
\begin{equation*}
N_{H_{0}}^{-}Z_{H_{0}}N_{H_{0}}^{+}=N_{H_{0}}^{-}N_{H_{0}}^{+}Z_{H_{0}}
\end{equation*}%
is open and dense in  $G$, where $Z_{H_{0}}$ is the centralizer of  $H_{0}$
in $G$ and the product
\begin{equation*}
\left( y,x,h\right) \in N_{H_{0}}^{-}\times N_{H_{0}}^{+}\times
Z_{H_{0}}\mapsto yxh\in N_{H_{0}}^{-}N_{H_{0}}^{+}Z_{H_{0}}
\end{equation*}%
is a diffeomorphism.

Consider now the adjoint orbit $\mathcal{O}\left( H_{0}\right) =%
\mathrm{Ad}\left( G\right) H_{0}=G/Z_{H_{0}}$. Then, the subset  $%
\mathrm{Ad}\left( N_{H_{0}}^{-}N_{H_{0}}^{+}Z_{H_{0}}\right) H_{0}$,
denoted simply by $N_{H_{0}}^{-}N_{H_{0}}^{+}Z_{H_{0}}\cdot H_{0}$
satisfies
\begin{equation*}
N_{H_{0}}^{-}N_{H_{0}}^{+}Z_{H_{0}}\cdot
H_{0}=N_{H_{0}}^{-}N_{H_{0}}^{+}\cdot H_{0}
\end{equation*}%
since $Z_{H_{0}}\cdot H_{0}=H_{0}$. Given that  $N_{H_{0}}^{-}N_{H_{0}}^{+}Z_{H_{0}}$
is open and dense in $G$ it follows that $N_{H_{0}}^{-}N_{H_{0}}^{+}\cdot H_{0}$
is open and dense in  $\mathcal{O}\left( H_{0}\right) $. Moreover, on one hand
the map
\begin{equation*}
\left( y,x\right) \in N_{H_{0}}^{-}\times N_{H_{0}}^{+}\mapsto yx\cdot
H_{0}\in N_{H_{0}}^{-}N_{H_{0}}^{+}\cdot H_{0}
\end{equation*}%
is a diffeomorphism, and on the other hand, $\exp :\mathfrak{n}_{H_{0}}^{\pm }\rightarrow N_{H_{0}}^{\pm
}$ is a diffeomorphism. Therefore, 
\begin{equation*}
\left( Y,X\right) \in \mathfrak{n}_{H_{0}}^{-}\times \mathfrak{n}%
_{H_{0}}^{+}\mapsto e^{\mathrm{ad}\left( Y\right) }e^{\mathrm{ad}\left(
X\right) }\cdot H_{0}\in N_{H_{0}}^{-}N_{H_{0}}^{+}\cdot H_{0}
\end{equation*}%
defines a coordinate system for  $\mathcal{O}\left( H_{0}\right) $, around
 $H_{0}$ (whose codomain is open and dense).

The singularities of the 
 Lefschetz fibration are  $e_{j}\otimes \varepsilon _{j}$, or equivalently, 
the points  $w\cdot H_{0}$ with $w=\left( 1j\right) \in \mathcal{W}$.
For the singularity  $e_{j}\otimes \varepsilon _{j}=w\cdot H_{0}$ the
algebras $\mathfrak{n}_{w\cdot H_{0}}^{\pm }$ and the groups $N_{w\cdot
H_{0}}^{\pm }=\exp \mathfrak{n}_{w\cdot H_{0}}^{\pm }$ are given by

\begin{itemize}
\item $\mathfrak{n}_{w\cdot H_{0}}^{+}$ consists of the matrices  with 
nonzero entries only at row  $j$ ($0$'s in the diagonal) whereas  $%
N_{w\cdot H_{0}}^{+}$ consists of the same   matrices but  with  $1$'s in the diagonal.

\item $\mathfrak{n}_{w\cdot H_{0}}^{-}$ 
consists of  matrices  with 
nonzero entries only at 
column $j$  ($0$'s in the  diagonal), whereas   $N_{w\cdot H_{0}}^{-}$ consists of the 
same matrices but with  $1$'s in the  diagonal.
\end{itemize}

For each index $j$ the open Bruhat cell $\sigma_j = N_{(1j)H_0}^-[e_j]$ is the set 
of  subspaces that are not contained in $V_j = \mathrm{span}\{e_1, \dots, \widehat{e_j}, \dots, e_n\}$.

\begin{proposition} The domain of the coordinate system for index  $j$ is given by 
$N_{(1j)H_0}^- N_{(1j)H_0}^+ \cdot (1j) H_0 = \{([v],V) \in \mathbb P^{n-1}\times Gr_{\varepsilon -1}(n) : [v]\in \sigma_j, v \notin V\}.$
This set coincides with the set of projections in the adjoint orbit of $e_1 \otimes \varepsilon_1$ whose image belongs to $\sigma_j$.
\end{proposition}

\begin{proof} $N_{(1j)H_0}^+ \cdot ([e_j],V_j) = \{[e_j]\} \cdot N_{(1j)H_0}^+$ where $N_{(1j)H_0}^+$ is the set of subspaces 
in $Gr_{n-1}(n)$ which do not contain $e_j$. Therefore if $n \in N_{(1j)H_0}^-$ then $n \left( \{[e_j]\} \cdot N_{(1j)H_0}^+ \cdot V_j\right) $
is the set of subspaces that do not contain $n [e_j]$.
\end{proof}

\begin{corollary} \label{corcharts}The domain $D_j = N_{(1j)H_0}^- N_{(1j)H_0}^+ \cdot (1j) H_0 $ of the chart for the index $j$ 
is Zariski open.
\end{corollary}

\begin{proof} In the adjoint orbit $\mathcal O = \mathcal O (e_1 \otimes \varepsilon_1)$ of $e_1 \otimes \varepsilon_1$ in $\mathbb C^n \times (\mathbb C^n)^*$
the domain $N_{(1j)H_0}^- N_{(1j)H_0}^+ \cdot (1j) H_0$ is given by the elements $v \notin V_j$, that is,  $\varepsilon_j(v)\neq 0.$
Let $ \mathcal O_k = \{v \otimes \varepsilon \in \mathcal O : \varepsilon(e_j) \neq 0\}.$ 
Clearly $\mathcal O = \cup_k \mathcal O_k$ therefore $D_j = \cup_k D_j \cap  \mathcal O_k .$ However
$$D_j \cap \mathcal O_k = \{ v \otimes \varepsilon \in \mathcal O : \varepsilon_j (v) \otimes \varepsilon(e_k) =0\}$$
$$=  \{ v \otimes \varepsilon \in \mathcal O :  \mathrm{tr} ((v\otimes \varepsilon)(e_k \otimes \varepsilon_j))\neq 0\}.$$
Since  $v \otimes \varepsilon \in \mathcal O \mapsto  \mathrm{tr} ((v\otimes \varepsilon)(e_k \otimes \varepsilon_j))$ is the restriction 
to $\mathcal O$ of a linear map, it follows that $D_j \cap \mathcal O_k$ is Zariski open, and thus so is $D_j$.
\end{proof}

\begin{remark}
Note that the complement of $D_j$ is the set of zeros of the polynomial $\sum_k ((\varepsilon_j(v) \varepsilon(e_k))^2$.
\end{remark}

We may restate corollary \ref{corcharts} as:

\begin{corollary}\label{bru}
The domains of the parametrizations  $D_j$ corresponding  the to  Bruhat cells  are open and dense in  $\mathcal{O}_\mu$.
\end{corollary}

\section{The potential viewed as a rational map}\label{pot}

Once the adjoint orbit has been compactified to a projective variety, 
we can no longer consider the potential as a holomorphic map, not even if we enlarge the target to 
$\mathbb P^1$. For the case of the minimal flag,  
 $\mathcal O(H_0)$ gets compactified to a product of projective spaces. 
The following 2 elementary lemmas show that the potential can not be extended holomorphically to the compactification.

\begin{lemma}\label{pot1}
Let $n>1$. Then any holomorphic map $\mathbb P^n \rightarrow \mathbb P^1$ is constant.
\end{lemma}

\begin{proof} Consider a holomorphic map $g \colon \mathbb P^n \rightarrow \mathbb P^1$ and let 
$X_1 = g^{-1}(p_1)$ and $X_2 =g^{-1}(p_2)$ be two of its fibers. Then by Bezout's theorem
$X_1\cap X_2 \neq 0$ therefore $p_1=p_2$.
\end{proof}

\begin{lemma}\label{pot2} Let $n>1$.
Then  any holomorphic map $  \mathbb P^n \times { \mathbb {P}^n}^\ast  \rightarrow  \mathbb P^1$ is 
constant.
\end{lemma}

\begin{proof}
Suppose $f \colon  \mathbb P^n \times { \mathbb {P}^n}^\ast  \rightarrow  \mathbb P^1$ is holomorphic, and 
take $p \in \mathbb {\mathbb {P}^n}^\ast$, then  the restriction $f\vert_{\mathbb P^n \times \{p\}}$ is holomorphic, thus constant
by lemma \ref{pot1}. Hence, $f$ factors through the projection  $ \mathbb P^n \times { \mathbb {P}^n}^\ast \rightarrow  { \mathbb {P}^n}^\ast $ and induces a holomorphic map ${ \mathbb {P}^n}^\ast  \rightarrow  \mathbb P^1$
which also is constant by lemma \ref{pot1}. Thus $f$ is constant.
\end{proof}

Consequently, we aim to extend the potential to the projectivization as a rational map. 
This can be done as follows. Set $V= \mathbb C^n$.

\begin{theorem}\label{rat}
 The rational function on  $V\otimes V^{\ast }$ that
coincides with  the potential $f_{H}$ on $\mathcal{O}\left( v_{0}\otimes \varepsilon
_{0}\right) $ 
is given by 
\begin{equation*}
R_{H}\left( A\right) =\frac{\mathrm{tr}\left( A\rho _{\mu }\left( H\right)
\right) }{\mathrm{tr}\left( A\right) }
\end{equation*}%
for $A\in V\otimes V^{\ast }=\mathrm{End}\left( V\right) $.
\end{theorem}

\begin{proof}
\begin{enumerate}
\item \label{1} Given 2 vector spaces   $V$ and $W$ let $\mathbb{P}\left(
V\right) $ and $\mathbb{P}\left( W\right) $ be the corresponding projective spaces. Then,
 $\mathbb{P}\left( V\right) \times \mathbb{P}\left( W\right) $  is in bijection with
the subset of  $\mathbb{P}\left( V\otimes
W\right) $ of subspaces generated by decomposable elements $%
v\otimes w$, $v\in V$ and $w\in W$.
The bijection is given by
\begin{equation*}
\left( \langle v\rangle ,\langle w\rangle \right) \in \mathbb{P}\left(
V\right) \times \mathbb{P}\left( W\right) \mapsto \langle v\otimes w\rangle
\in \mathbb{P}\left( V\otimes W\right) .
\end{equation*}

\item The flag  $\mathbb{F}_{H_{\mu }}$ gets identified with the projective 
orbit of the space of maximal weight $V_{\mu }=\langle v_{0}\rangle
\in \mathbb{P}\left( V\right) $ ($V$ = the representation space).
The dual flag $\mathbb{F}_{H_{\mu }}^{\ast }$ gets identified to the projective 
orbit of the space of minimal weight  $V_{\mu ^{\ast }}=\langle
\varepsilon _{0}\rangle \in \mathbb{P}\left( V^{\ast }\right) $. 

\item  The adjoint orbit  $\mathcal{O}\left( H_{\mu }\right) $ gets
identified with the open orbit  in  $\mathbb{F}_{H_{\mu }}\times \mathbb{F%
}_{H_{\mu }}^{\ast }$ by the diagonal action of $G$. 
Via the bijection of item \ref{1} the compactification corresponds to  the 
projectivization of the orbit $\mathcal{O}\left( v_{0}\times
\varepsilon _{0}\right) $ of $v_{0}\otimes \varepsilon _{0}\in V\otimes
V^{\ast }$. 

\item The potential  $f_{H}$ on $\mathcal{O}\left( v_{0}\times \varepsilon
_{0}\right) $ can be written as 
\begin{equation}
f_{H}\left( v\otimes \varepsilon \right) =\varepsilon \left( \rho _{\mu
}\left( H\right) v\right) =\mathrm{tr}\left( \left( v\otimes \varepsilon
\right) \rho _{\mu }\left( H\right) \right)   \label{forfaga}
\end{equation}%
where $\rho _{\mu }$ is the representation on  $V$. 

\item The function  $f_{H}$ of (\ref{forfaga}) does not projectivize,
that is, it does not induce a function on  $\mathbb{P}\left( V\otimes
V^{\ast }\right) $ since it is linear (homogeneous of degree $1$). To
projetivize the potential we must divide $f_{H}$ by a linear function that is 
constant on the orbit $\mathcal{O}\left( v_{0}\otimes \varepsilon
_{0}\right) $, therefore obtaining a  rational homogeneous function (of degree $0$%
) which coincides with $f_{H}$ on the orbit and projectivizes.

\item A linear function that can do the job is $\mathrm{tr}%
\left( v\otimes \varepsilon \right) =\varepsilon \left( v\right) $. This 
linear functional is constant $=1$ on  $\mathcal{O}\left( v_{0}\otimes
\varepsilon _{0}\right) $, since if  $v\otimes \varepsilon \in \mathcal{O}%
\left( v_{0}\otimes \varepsilon _{0}\right) $ then there exists  $g\in G$ such 
that  
\begin{equation*}
v\otimes \varepsilon =\rho _{\mu }\left( g\right) v_{0}\otimes \rho _{\mu
}^{\ast }\left( g\right) \varepsilon _{0}=\rho _{\mu }\left( g\right)
v_{0}\otimes \varepsilon _{0}\circ \rho _{\mu }\left( g^{-1}\right) ,
\end{equation*}%
thus $\varepsilon \left( v\right) =\varepsilon _{0}\circ \rho _{\mu
}\left( g^{-1}\right) \left( \rho _{\mu }\left( g\right) v_{0}\right)
=\varepsilon _{0}\left( v_{0}\right) =1$. 

\item Therefore, the rational function on  $V\otimes V^{\ast }$ that
coincides with  $f_{H}$ on $\mathcal{O}\left( v_{0}\otimes \varepsilon
_{0}\right) $ and projectivizes is given by 
\begin{equation*}
R_{H}\left( A\right) =\frac{\mathrm{tr}\left( A\rho _{\mu }\left( H\right)
\right) }{\mathrm{tr}\left( A\right) }
\end{equation*}%
for $A\in V\otimes V^{\ast }=\mathrm{End}\left( V\right) $.
\end{enumerate}
\end{proof}

%

\section{Algebraic compactifications and the conjecture of \cite{BG}}

\noindent{\bf The orbit  $\mathcal O_\mu$.}  We can also  compactify of  $\mathcal O_\mu$  from 
an algebraic point of view. 
Let $X_n = \mathcal O(H_0) $ for $H_0=\Diag(n, -1, \dots, -1)$.
Then $A$ belongs to $X_n$ if and only if it satisfies the equations of the minimal polynomial $(A-nI)(A+I)=0$.
To compactify  to a projective variety $\overline{X}_n$ we add  an extra variable $t$ and homogenise the
 equations to $(A-ntId)(A+tId) =0$. The set $F_n = \overline{X}_n \setminus X_n$ has a complete description. Taking $t=0$
we get that it is gives by the system $A^2=0$. Set-theoretically in $\mathfrak{sl}(n+1, \mathbb C)$ the system  $A^2=0$
 defines all trace-zero matrices such that $t^2$  divides the minimal polynomial.
In \cite{BG} it is proven that this algebraic compactification produces the Segre embbeding (assuming the additional hypothesis of 
smoothness) and it is proven computationally using Macaulay 2 for the cases of $\mathfrak{sl}(n,\mathbb C)$ for $n<10$
that this algebraic compactification does produce the Segre embedding. However, it is left as a conjecture to show that this works in full generality.

The following result follows from the explicit calculations presented in the next section, and solves
the conjecture of \cite{BG} in the affirmative:

\begin{theorem}\label{seg}
The embbeding  $\mathcal{O}_\mu \hookrightarrow \mathbb{P}^n\times  {\mathbb{P}^n}^*$ obtained by Lie 
theoretical methods agrees with the Segre embbeding obtained algebro-geometrically by homogenisation of the ideal cutting out 
the orbit $\mathcal O_\mu$ as an affine variety in $\mathfrak{sl}(n+1)$.
\end{theorem}


\noindent{\bf Other adjoint orbits.}  As proved in \cite[Sec. 3]{GGSM2}  the orbit of any  trace zero diagonalizable matrix $H_0\ne 0$ embeds   as an open dense  subset
of $\mathbb F\times \mathbb F^*$ with $\mathbb F$ a certain flag. This is  the best compactification, but  we wish to compare with other
compactifications  obtained via algebraic methods.
To compactify algebraically  \cite{BCG}  used the process of homogenization of the ideal defining the orbit inside its Lie algebra. 
A matrix $A$ belongs to the adjoint orbit  $H_0$ if and only if it satisfies  the equations 
of  the minimal polynomial of $H_0$. Taking the entries of the minimal polynomial determines an ideal $I$
cutting out $\mathcal O(H_0)$ as an affine variety inside $\mathfrak{sl}(n, \mathbb C)$. 
We may then obtain a compactification by homogenizing the ideal $I$. In general the resulting compactification 
will be very singular, see
\cite[Sec. 6]{BCG}. So, it is not possible to generalize the conjecture of \cite{BG} for all semisimple 
adjoint orbits.

\section{Lie theoretical  compactification and the Segre embedding}

In this section we present the explicit Lie theoretical calculation of the Segre embedding, first with the case $n=3$ and then 
the general case. 
\vspace{5mm}

 \noindent{\bf The case of $\mathfrak{sl}(3,\mathbb{C})$}
 Let $G= SL(3,\mathbb{C})$ and let  $g$ be an element of  $G$. We write $g$ as 
 \begin{equation}\label{sl3-generic}
 g=\left(
\begin{array}{ccc}
 a_{11} & a_{12}  & a_{13}   \\
 a_{21}    & a_{22}     & a_{23}      \\
 a_{31}       & a_{32}        & a_{33}         \\
\end{array}
\right), 
 \end{equation}
where $$\det g =a_{11}   a_{22} a_{33}       -a_{11} a_{23}      a_{32}+a_{21}    a_{13}   a_{32}-a_{21}    a_{33}         a_{12}-a_{31}       a_{13}   a_{22}+a_{31}       a_{23}      a_{12}=1.$$ 
 
 The  inverse of $g$ is given by
 \begin{equation}
 g^{-1}=\left(
\begin{array}{ccc}
 a_{33}         a_{22}-a_{23}      a_{32}        & a_{13}   a_{32}-a_{33}         a_{12}  & a_{23}      a_{12}-a_{13}   a_{22}     \\
 a_{31}       a_{23}-a_{21}    a_{33}         & a_{11} a_{33}-a_{31}       a_{13}   & a_{21}    a_{13} -a_{11} a_{23}      \\
 a_{21}    a_{32}-a_{31}       a_{22}     & a_{31}       a_{12}-a_{11} a_{32}        & a_{11} a_{22}-a_{21}    a_{12}  \\
\end{array}
\right) \text{.}
 \end{equation}

Let us describe the orbit  $G\cdot (v_0\otimes \varepsilon_0 )=\rho(g)v_0\otimes \rho^*(g)\varepsilon_0$.
Recall that $v_0=(1,0,0)^T$ and $\varepsilon_0=(1,0,0)$. The actions are
\begin{equation}
\rho(g)v_0=gv_0= 
\left(
\begin{array}{c}
 a_{11} \\
 a_{21}    \\
 a_{31}       \\
\end{array}
\right),
\end{equation} and
\begin{equation}\label{dual-action}
\rho^*(g)\varepsilon_0=\varepsilon_0\circ       g^{-1}=
\left(
\begin{array}{ccc}
 a_{33}         a_{22}-a_{23}      a_{32}        & a_{13}   z-a_{33}         a_{12}  & a_{23}      a_{12}-a_{13}   a_{22}     \\
\end{array}
\right).
\end{equation}
Therefore,
\begin{equation} \label{mat-sl3}
\rho(g)v_0\otimes \rho^*(g)\varepsilon_0=
\end{equation}
$$\left(
\begin{array}{ccc}
 a_{11} a_{33}         a_{22}-a_{11} a_{23}      a_{32}        & a_{11} a_{13}   z-a_{11} a_{33}         a_{12}  & a_{11} a_{23}      a_{12}-a_{11} a_{13}   a_{22}     \\
 a_{21}    a_{33}         a_{22}-a_{21}    a_{23}      a_{32}        & a_{21}    a_{13}   z-a_{21}    a_{33}         a_{12}  & a_{21}    a_{23}      a_{12}-a_{21}    a_{13}   a_{22}     \\
 a_{31}       a_{33}         a_{22}-a_{31}       a_{23}      a_{32}        & a_{31}       a_{13}   z-a_{31}       a_{33}         a_{12}  & a_{31}       a_{23}      xa_{12}-a_{31}       a_{13}   a_{22}     \\
\end{array}
\right)\text{.}
$$
The eigenvalues of  matrix 
$(\ref{mat-sl3})$ are:
\begin{itemize}
\item $1$ associated to the vector $\mu=(a_{11},a_{21},a_{31})$ 
and
\item $0$ (zero) associated to the vectors $\xi=(a_{12},a_{22},a_{23})$ and $\eta=(a_{13},a_{23},a_{33})$. 
\end{itemize}

Since the determinant of matrix  $(\ref{sl3-generic})$ is nonzero, we have that the line generated by      $\mu$ 
is transversal to the plane generated by
 $\xi$ and $\eta$ (this is the geometric description of the adjoint orbit).
Using the moment map, we verify that the orbit of the tensor product is isomorphic  
to  $\mathrm{Ad}(G)\cdot H_0$, for $H_0$ chosen apropriately (as a multiple of  $\Diag(2,-1,-1)$).

Thus, we obtain and embedding of the minimal orbit 
 $$\varphi:G\cdot (v_0\otimes \varepsilon_0 )\rightarrow \mathbb{P}^2\times G_2(\mathbb{C}^3),$$ given by 
  $\varphi(g\cdot (v_0\otimes \varepsilon_0 ))=(\mathrm{span}\{\mu\},\mathrm{span}\{ \xi,\eta \}   )$.  

An identification between $G_2(\mathbb{C}^3)$ and $\mathbb{P}^2$ is obtained by taking each 2-plane $P$ in $G_2(\mathbb{C}^3)$ to the line  $\ell_P$ generated by the normal vector  $P$. 
Explicitly, if $P=\mathrm{span}\{\xi,\eta\}$, then $\ell_P$ is generated by the vector 
\begin{equation} \label{lp}
(a_{33}         a_{22}         -a_{23}      a_{32}         )\vec{i}+ (a_{13}   a_{32}         -a_{33}         a_{21})\vec{j} + (a_{23}      a_{12}-a_{13}   a_{22}         )\vec{k}
\end{equation}
$$=(a_{33}         a_{22}         -a_{23}      a_{32}         ,a_{13}   a_{32}         -a_{33}         a_{21},a_{23}      a_{21}-a_{13}   a_{22}     ) \text{.}
$$
Observe that $(\ref{lp})$ recovers the result of  $(\ref{dual-action})$.

Note that if a vector $\mu$ belongs to the plane 
generated by  $\{\xi, \eta\}$ then $\mu$ is orthogonal to the vector described in $(\ref{lp})$, that is , 
\begin{equation}
(a_{11},a_{21},a_{31})\cdot (a_{33}         a_{22}         -a_{23}      a_{32}         ,a_{13}   a_{32}         -a_{33}         a_{21},a_{23}      a_{21}-a_{13}   a_{22}     ) =
\end{equation}
$$
-a_{11} a_{23}      a_{32}         +a_{11} a_{33}         a_{22}         +a_{21}    a_{13}   a_{32}         -a_{21}    a_{33}         a_{12}-a_{31}       a_{13}   a_{22}         +a_{31}       a_{23}      a_{12}  = 0.
$$
This expression corresponds to the determinant of a 3x3 matrix (contained in the complement of the orbit inside $\mathbb{P}^2\times G_2(\mathbb{C}^3)$, that is,  representing the case of a line contained in a plane).

Using the previous identification, we can now obtain the Segre embedding by taking the composite
\begin{equation}\label{composta-sl3}
G\cdot (v_0\otimes \varepsilon_0 )\rightarrow \mathbb{P}^2\times G_2(\mathbb{C}^3)\rightarrow \mathbb{P}^2\times \mathbb{P}^2 \rightarrow \mathbb{P}^8.
\end{equation}

The image of  $g\cdot(v_0\otimes \varepsilon_0)$ by the composite $(\ref{composta-sl3})$ in $\mathbb{P}^8$ 
has  homogeneous coordinates that are the same as the entries of matrix $(\ref{mat-sl3})$.
\vspace{5mm}

{\bf The rational map: } We  describe $f_H$, for  $H=(3,-2,-1)$.  (We  chose this instead of  $H=(1,0,-1)$ to avoid 
the vanishing of monomials that would be caused by  the zero, but the same method applied to any choice of $H$). 
The rational map $R_H$ associated to  $f_H$ is given by 
\begin{equation}\label{rat-sl3}
R_H(v\otimes \varepsilon)= \frac{\mathrm{tr}((v\otimes \varepsilon) \rho(H))}{\mathrm{tr}(v\otimes \varepsilon)}=
\end{equation}
$$\frac{3a_{11}a_{33}a_{22}         -3a_{11}a_{23}a_{32}         -2a_{21}a_{13}a_{32}         +2a_{21}a_{33}a_{12}-a_{31}a_{23}a_{12}+a_{31}a_{13}a_{22}         }{a_{11} a_{33}         a_{22}     -a_{11} a_{23}      a_{32}        +a_{21}    a_{13}   a_{32}      -a_{21}    a_{33}         a_{12} +a_{31}       a_{23}      a_{12}  -a_{31}       a_{13}   a_{22}         }.
$$

The denominator of $(\ref{rat-sl3})$ is the  determinant, which equals  1 if the point belongs to the orbit,
 and vanishes if the point belongs to the complement of the orbit. Thus, for points in the orbit  $R_H$ coincides with $f_H$ (up to a  constant multiple). Consequently, we can use the composite $(\ref{composta-sl3})$ to define a map to $\mathbb P^1$,
  factoring through the  Segre embedding:
\begin{equation}
 \mathbb{P}^2\times G_2(\mathbb{C}^3)\rightarrow \mathbb{P}^1, 
\end{equation} 
by
\begin{equation}
(\mathrm{span} \{\mu\}, \mathrm{span}\{\xi,\eta\} )\mapsto 
\end{equation}
$$[3a_{11}a_{33}a_{22}         -3a_{11}a_{23}a_{32}         -2a_{21}a_{13}a_{32}         +2a_{21}a_{33}a_{12}-a_{31}a_{23}a_{12}+a_{31}a_{13}a_{22}        :$$
$$ a_{11} a_{33}         a_{22}     -a_{11} a_{23}      a_{32}        +a_{21}    a_{13}   a_{32}         -a_{21}    a_{33}         a_{12}
+a_{31}       a_{23}      a_{12}-a_{31}       a_{13}   a_{22}           ] .
$$
\\
\\

\noindent{\bf The general case: $\mathfrak{sl}(n, \mathbb C)$}

For an  $n\times n$ matrix $A$, we denote by $A(i|j)$ 
the matrix obtained by removing the  $i$-th row and the  $j$-th column of $A$. 
Recall that the $(i,j)$-cofactor  of  $A$ is the scalar 
\begin{equation}
C_{ij}=(-1)^{i+j}\det A(i|j).
\end{equation} 
We denote by $C=C_{ij}$ the matrix of cofactors.
The classical adjoint of $A$ is the transpose of the matrix of cofactors: 
\begin{equation}
(\adj A)_{ij}=C_{ji}.
\end{equation}
We will use the following 2 well known properties of the classical adjoint:
\begin{equation}
\sum_{i=1}^n{A_{ik}(\adj A)_{ji}}=\delta_{kj}\det A;
\end{equation}

\begin{equation}
A\,(\adj A)=(\det A)\,\id.
\end{equation}
In particular, for a fixed $j$ we obtain $\sum_{i=1}^n{A_{ij}(\adj A)_{ji}}=\det A$ (expansion in cofactors with respect to  column 
 $j$).
\\

{\bf The general  Segre embedding:} Let $G=SL(n,\mathbb{C})$ and $g=(a_{ij})\in G$. We denote by  $w_i=(a_{1i},\ldots a_{ni})$ the column vectors of  $g$. Since $\det g =1$, we have that $g^{-1}=\adj g$. 

Let  $v_0=(1,0,\ldots,0)\in \mathbb{C}^n$ and  $\varepsilon_0=(1,0,\ldots,0)\in (\mathbb{C}^n)^*$. We describe the orbit 
 $G\cdot (v_0\otimes \varepsilon_0)$. We have:
\begin{equation}
\rho(g)v_0=gv_0=(a_{11},a_{21},\ldots,a_{n1})=w_1;
\end{equation}
\begin{equation}
\rho^*(g)\varepsilon_0=\varepsilon_0 \circ g^{-1}=\varepsilon_0 \circ \adj g=((\adj g)_{11},(\adj g)_{12},\ldots,(\adj g)_{1n} ).
\end{equation}

Therefore,
\begin{equation}\label{mat-tensor}
\rho(g)v_0\otimes \rho^*(g)\varepsilon_0 = M=M_{ij}= a_{i1} (\adj g)_{1j}.
\end{equation}
Observe that
\begin{equation}\label{det1}
\mathrm{tr} M = \sum_{i=1}^n {M_{ii}} = \sum_{i=1}^n{ a_{i1} (\adj g)_{1i} }=\det g=1  \text{.}
\end{equation}
We can describe explicitly the kernel and image of $M$:
\begin{eqnarray*}
M(w_1)&=&
\left(
\begin{array}{cccc}
a_{11} (\adj g)_{11}& a_{11} (\adj g)_{12}& \ldots & a_{11} (\adj g)_{1n}\\
\vdots &\vdots & & \vdots\\
a_{n1} (\adj g)_{11}& a_{n1} (\adj g)_{12}& \ldots & a_{n1} (\adj g)_{1n}
\end{array}
\right)
\left(
\begin{array}{c}
a_{11}\\
\vdots \\
a_{n1}
\end{array}
\right)\\ \\
&=&
\left(
\begin{array}{c}
a_{11}\{a_{11} (\adj g)_{11} +a_{21} (\adj g)_{12}+\ldots +a_{n1} (\adj g)_{1n} \}\\
\vdots\\
a_{n1}\{a_{11} (\adj g)_{11} +a_{21} (\adj g)_{12}+\ldots +a_{n1} (\adj g)_{1n} \}
\end{array}
\right) \\ \\
&=& 
\left(
\begin{array}{c}
a_{11}\\
\vdots \\
a_{n1}
\end{array}
\right)=w_1.
\end{eqnarray*}
\\
Hence, $w_1$ is an eigenvector associated to the eigenvalue $1$.

On the other hand, 
\begin{eqnarray*}
M(w_2)&=&
\left(
\begin{array}{cccc}
a_{11} (\adj g)_{11}& a_{11} (\adj g)_{12}& \ldots & a_{11} (\adj g)_{1n}\\
\vdots &\vdots & & \vdots\\
a_{n1} (\adj g)_{11}& a_{n1} (\adj g)_{12}& \ldots & a_{n1} (\adj g)_{1n}
\end{array}
\right)
\left(
\begin{array}{c}
a_{12}\\
\vdots \\
a_{n2}
\end{array}
\right)\\ \\
&=&
\left(
\begin{array}{c}
a_{11}\{a_{12} (\adj g)_{11} +a_{22} (\adj g)_{12}+\ldots +a_{n2} (\adj g)_{1n} \}\\
\vdots\\
a_{n1}\{a_{12} (\adj g)_{11} +a_{22} (\adj g)_{12}+\ldots +a_{n2} (\adj g)_{1n} \}
\end{array}
\right) \\ \\
&=& 
\left(
\begin{array}{c}
0\\
\vdots \\
0
\end{array}
\right).
\end{eqnarray*}
\\
Hence, $w_2$ is in the kernel of $M$. Analogously, we verify that $w_2,\ldots,w_n$ are in the kernel of $M$ 
(and therefore they are eigenvectors associated to the zero eigenvalue).
As a consequence, we obtain the embedding
\begin{equation}
\varphi:G\cdot (v_o\otimes \varepsilon_0) \rightarrow \mathbb{P}(\mathbb{C}^n)\times G_{n-1}(\mathbb{C}^n)
\end{equation}
given by $\varphi(g\cdot (v_0\otimes \varepsilon_0 ))=(\mathrm{span}\{w_1\},\mathrm{span}\{ w_2,\dots, w_n \} )$.

Let $P$ be the hyperplane generated by $w_2,\ldots,w_n$. Denote by $\xi_P\in (\mathbb{C}^n)^*$ 
the linear functional associated to  $P$ (that is, $P$ is in the kernel of $\xi_P$). Direct calculations show that
 $$\xi_P=((\adj g)_{11},(\adj g)_{12},\ldots,(\adj g)_{1n} ).$$
The correspondence  $P\mapsto\xi_P$ gives the identification $$G_{n-1}(\mathbb{C}^n)\rightarrow \mathbb{P}( (\mathbb{C}^n)^* ).$$
The Segre embedding of the minimal orbit is the composite
\begin{equation}
G\cdot (v_0\otimes \varepsilon_0 )\rightarrow \mathbb{P}(\mathbb{C}^n )   \times G_{n-1}(\mathbb{C}^n)\rightarrow \mathbb{P}(\mathbb{C}^n ) \times \mathbb{P}((\mathbb{C}^n)^* ) \rightarrow \mathbb{P}^{n^2-1}.
\end{equation}
The coordinates of the image of this composite in $\mathbb{P}^{n^2-1}$ are the entries of matrix $(\ref{mat-tensor})$.

Observe that the complement of the adjoint orbit in  $\mathbb{P}(\mathbb{C}^n ) \times \mathbb{P}((\mathbb{C}^n)^* )$ 
is the   {\it incidence correspondence variety} $\Sigma$'' (see \cite[Ex. 6.12]{Ha}) given by
\begin{equation}
\Sigma = \{  (w,\xi ): \, \langle w,\xi \rangle =0   \}\subset \mathbb{P}(\mathbb{C}^n ) \times \mathbb{P}((\mathbb{C}^n)^* ) \text{.}
\end{equation}

\noindent{\bf The rational map: }
\\

Let $H=\Diag(\lambda_1,\ldots,\lambda_n)\in \mathfrak{h}$, with $\lambda_1>\ldots>\lambda_n $ and
 $\lambda_1 + \ldots +\lambda_n=0$ (where $\mathfrak{h}$ is the  Cartan subalgebra of  $\mathfrak{sl}(n,\mathbb{C})$).

We  describe a rational map (factored through the  Segre embbeding), that coincides with the potential $f_H$ 
on the adjoint orbit. Such a rational map is given by 
\begin{equation}
 \psi:\mathbb{P}^{n-1}\times G_{n-1}(\mathbb{C}^n)\rightarrow \mathbb{P}^1, 
\end{equation} 
\begin{equation}
\psi([v],[\varepsilon]) = \frac{\tr((v\otimes \varepsilon) \rho(H))}{\tr(v\otimes \varepsilon)}=\frac{\sum_{i=1}^n{\lambda_i a_{i1} (\adj g)_{1i} }}{\sum_{i=1}^n{ a_{i1} (\adj g)_{1i} }},
\end{equation}
where the identification $([v],[\varepsilon])\mapsto v\otimes \varepsilon$  is described in  \cite[Sec. 4.2]{GGSM2}. Observe that if  $([v],[\varepsilon])$ belongs to the adjoint orbit, then $\tr(v\otimes\varepsilon)=1$ (see eq. \ref{det1}). 
Furthermore, the complement of the orbit is the  {\it incidence correspondence variety} $\Sigma$', that is, 
the set of pairs $(\ell,P)$ such that
\begin{equation}
0\subset \ell\subset P \subset\mathbb{C}^n,
\end{equation}
where $P$ is a hyperplane in  $\mathbb{C}^n$ and $\ell\subset P$ is a line. 
The variety  $\Sigma$ is classically denoted in Lie theory as the flag  manifold $\mathbb{F}(1,n-1)$. 

\end{document}